\documentclass[amssymb,amsfonts,refcheck,12pt,verbatim,righttag]{amsart}

\usepackage{amssymb,mathrsfs}

\usepackage{amssymb}
\usepackage{graphicx}
\usepackage{graphics,color}

\usepackage{amsfonts}
\usepackage{color}
\usepackage{amssymb}
\usepackage{cancel}
\usepackage{soul}

 \numberwithin{equation}{section} 
\theoremstyle{plain}

\newtheorem{thm}{Theorem}[section]
\newtheorem{lem}[thm]{Lemma}
\newtheorem{pro}[thm]{Proposition}
\newtheorem{cor}[thm]{Corollary}

\newtheorem{de}[thm]{Definition}

\def\R {{\Bbb R}}
\def\N {{\Bbb N}}
\def\Z {{\Bbb Z}}

\def\ba{{\bf a}}

\def\va{{\bf a}}

\newcommand{\bi}{\mathbf{i}}

\newcommand{\innerprod}[2]{\langle #1, #2 \rangle}
\providecommand{\norm}[2][]{\lVert#2\rVert_{#1}}
\providecommand{\abs}[1]{\lvert#1\rvert}
\newcommand{\euclid}[1][d]{\mathbb{R}^{#1}}

\usepackage[colorlinks,citecolor=blue,pagebackref,urlcolor=black,hypertexnames=false]{hyperref}

\begin{document}
\baselineskip 13.7pt
\title{Typical self-affine sets with non-empty interior}

\author{De-Jun FENG}
\address{
Department of Mathematics\\
The Chinese University of Hong Kong\\
Shatin,  Hong Kong
}
\email{djfeng@math.cuhk.edu.hk}

\author{Zhou Feng}
\address{
Department of Mathematics\\
The Chinese University of Hong Kong\\
Shatin,  Hong Kong
}
\email{zfeng@math.cuhk.edu.hk}


\date{}
\maketitle
\begin{center}
{\small \it Dedicated to the memory of Professor Ka-Sing Lau}
\end{center}

\begin{abstract} Let $T_1,\ldots, T_m$ be a family of $d\times d$ invertible real matrices with $\|T_i\| <1/2$  for $1\leq i\leq m$.
We provide some sufficient conditions on these matrices such that the self-affine set generated by the iterated function system $\{T_ix+a_i\}_{i=1}^m$ on $\R^d$ has non-empty interior for almost all $(a_1,\ldots, a_m)\in \R^{md}$.
\end{abstract}

\

\section{Introduction}
\label{S1}
In this paper, we provide some sufficient conditions  for a typical self-affine set to have non-empty interior.

Let us first introduce some necessary notation and definitions. By an {\it affine iterated function system} on $\R^d$  we mean a finite  family ${\mathcal F}=\{f_i\}_{i=1}^m$ of affine mappings from $\R^d$ to $\R^d$,
taking  the form
\begin{equation*}
\label{e-form}
f_i(x)=T_ix+a_i,\qquad i=1,\ldots, m,
\end{equation*}
where $T_i$ are contracting $d\times d$ invertible real  matrices and $a_i\in \R^d$.  It is  well known \cite{Hutchinson1981} that  there  exists a unique non-empty compact set $K\subset \R^d$ such that
$$
K=\bigcup_{i=1}^m f_i(K).
$$
We call $K$ the {\it attractor} of $\mathcal F$, or the {\em self-affine set} generated by ${\mathcal F}$.

In what follows,  let $T_1,\ldots, T_m$ be a fixed family of contracting $d\times d$ invertible real  matrices. Let  $\Sigma=\{1,\ldots, m\}^\N$ denote the symbolic space over the alphabet $\{1,\ldots, m\}$. Endow $\Sigma$ with the  product topology and let $\mathcal P(\Sigma)$ denote the space of  Borel probability measures on $\Sigma$.

For $\va = (a_1, \ldots, a_m) \in \R^{md}$, let $\pi^\ba: \Sigma \to \R^d$ be the coding map associated with the IFS $\{ f_i^{\ba}(x) = T_ix + a_i\}_{i=1}^m$, here we write $f_i^\ba$ instead of $f_i$ to emphasize its dependence of $\ba$. That is,
\begin{equation}
\label{e-pia}
\pi^\ba (\bi) = \lim_{n \to \infty} f^{\va}_{i_1} \circ \cdots \circ f^{\va}_{i_n}(0)
\end{equation}
for $\bi =(i_n)_{n=1}^\infty\in \Sigma$. Set $K^\ba=\pi^\ba(\Sigma)$. It is well known \cite{Hutchinson1981} that $K^\ba$ is the attractor of $\{f^\ba_i\}_{i=1}^m$.

In his seminal work \cite{Falconer1988},  Falconer  introduced a quantity associated to the matrices $T_1,\ldots, T_m$, nowadays usually called the {\em affinity dimension} $\dim_{\rm AFF}(T_1,\ldots, T_m)$ (see Definition~\ref{def:affinity}), which is always an upper bound for the upper box-counting dimension of $K^\ba$, and such that when
$\|T_i\|<1/2$  for all $1\leq i\leq m$,
then for ${\mathcal L}^{md}$-a.e.~$\ba\in \R^{md}$,   $$\dim_{\rm H}K^\ba=\dim_{\rm B}K^\ba=\min \{d, \dim_{\rm AFF}(T_1,\ldots, T_m)\}.$$
where $\dim_{\rm H}$ and $\dim_{\rm B}$ stand for the Hausdorff and box-counting dimensions respectively (see e.g.~\cite{Falconer2003} for the definitions). In fact, Falconer proved this with $1/3$ as the upper bound on the norms; it was subsequently shown by Solomyak \cite{Solomyak1998} that $1/2$ suffices. Later, Jordan, Pollicott and Simon \cite{JordanEtAl2007} further showed that if $\|T_i\|<1/2$ for all $i$ and $\dim_{\rm AFF}(T_1,\ldots, T_m)>d$, then
$K^\ba$ has positive Lebesgue measure for ${\mathcal L}^{md}$-a.e.~$\ba\in \R^{md}$.  We remark that the condition $\dim_{\rm AFF}(T_1,\ldots, T_m)>d$ is equivalent to $\sum_{i=1}^m |\det(T_i)|>1$, where $\det(T_i)$ denotes the determinant of $T_i$.

A question arises naturally that under which conditions on $T_1,\ldots, T_m$, $K^\ba$ has non-empty interior for ${\mathcal L}^{md}$-a.e.~$\ba\in \R^{md}$. Although this seems a rather fundamental question,  it has hardly been studied.

In this paper, we study the above question.  For a $d\times d$ real matrix $A$, let $\alpha_1(A)\geq \cdots \geq \alpha_d(A)$ denote the singular values of $A$, that is, $\alpha_1(A),\ldots, \alpha_d(A)$ are square roots of the eigenvalues of $A^*A$.  Here $A^*$ stands for the transpose of $A$. Write $\Sigma_n=\{1,\ldots, m\}^n$ for $n\in \N$ and  set $T_I=T_{i_1}\cdots T_{i_n}$ for $I=i_1\ldots i_n\in \Sigma_n$.   Define
\begin{equation}
\label{e-tT}
	t(T_{1}, \ldots, T_{m}) = \inf \left \{ t\geq 0 \colon \sup_{n\geq 1}   \sum_{I\in\Sigma_{n}} \alpha_d(T_I)^t |\det(T_I)| \leq 1 \right \}.
\end{equation}

The first result of the paper is the following.

\begin{thm}\label{thm:general}
	Assume that $ \norm{T_{i}} < 1/2  $ for $1\leq i\leq m $. Suppose $t(T_{1}, \ldots, T_{m}) > d$.  Then $ K^{\ba} $  has non-empty interior for $ \mathcal L^{md} $-a.e.~$ \ba  \in \euclid[md] $.
\end{thm}

It is easy to see that $t(T_{1}, \ldots, T_{m}) > d$ if and only if $\sum_{I\in \Sigma_n}\alpha_d(T_I)^d |\det(T_I)|>1$ for some $n\in \N$. As a direct corollary of Theorem \ref{thm:general}, we have the following.

\begin{cor}
\label{cor-1.2}
Assume that $ \norm{T_{i}} < 1/2  $ for $1\leq i\leq m $. Then  $ K^{\ba} $  has non-empty interior for $ \mathcal L^{md} $-a.e.~$ \ba  \in \euclid[md] $ provided that one of the following two conditions fulfills:
\begin{itemize}
\item[(i)] $\sum_{i=1}^{m}\alpha_{d}(T_{i})^{d}\abs{\det (T_{i})} > 1.$
\medskip
\item[(ii)] All $T_i$ are scalar multiples of orthogonal matrices, and $\sum_{i=1}^m |\det(T_i)|^2>1$.
\end{itemize}
\end{cor}

Next we provide an improvement of Theorem \ref{thm:general} in the special case when the matrices $T_1,\ldots, T_m$ commute.

\begin{thm}\label{thm:abelian}
	Assume that $ \norm{T_{i}} < 1/2  $ for $1\leq i\leq m $.  Moreover, suppose  that  $ T_{i}T_{j} = T_{j}T_{i} $ for all $ 1\leq i,j\leq m$, and  $ \sum_{i=1}^{m} \abs{\det (T_{i})}^{2} > 1 $.
	Then $ K^{\ba} $ has non-empty interior for $ \mathcal L^{md} $-a.e.~$ \ba \in \euclid[md]
	$.
\end{thm}

The above results provide some sufficient conditions on $(T_1,\ldots, T_m)$ such that $K^\ba$ has non-empty interior for almost all $\ba$. We don't know whether these conditions are sharp.  A natural conjecture is that the conditions of $\|T_i\|<1/2$ for  $1\leq i\leq m$ and $\sum_{i=1}^m |\det(T_i)|>1$ would suffice for $K^\ba$ to have the non-empty interior  for almost all $\ba$.

Now we address  some related works in the literature.  In \cite{Shmerkin2006} Shmerkin investigated a special class of overlapping affine IFSs $\Phi_{\alpha,\beta}=\{\phi_i(x,y)=(\alpha x +d_i, \beta y+d_i)\}_{i=1}^2$ on the plane with $d_1=0$ and $d_2=1$. He proved, among other things,  that there is an open region $V\subset \{(\alpha, \beta)\colon 0<\alpha<\beta<1,\; 2\alpha \beta>1\}$ such that for almost all $(\alpha,\beta)\in V$,  the attractor $K_{\alpha, \beta}$ of $\Phi_{\alpha, \beta}$ has non-empty interior.   
Later Dajani, Jiang and Kempton \cite{DJK2014} showed that there exists a number $C\approx 1.05^{-1}$ such that  $K_{\alpha, \beta}$ has non-empty interior for each pair $(\alpha, \beta)$ satisfying  $C<\alpha<\beta<1$.  The result of Dajani  et al. was subsequently improved and extended in \cite{HareSidorov2016, HareSidorov2017, Baker2020}.  We remark that the interior problem has also been extensively studied for integral self-affine sets (see e.g.~\cite{Bandt1991, LagariasWang1996, HeLauRao2003} and the references therein). Recall that an integral self-affine set  is the attractor of an affine IFS $\{Ax+a_i\}_{i=1}^m$ on $\R^d$ in which $A^{-1}$ is an integral expanding $d\times d$ matrix and $a_i\in \Z^d$ for all $1\leq i\leq m$. In \cite{HeLauRao2003}, He, Lau and Rao produced, among other things, a finite algorithm to determine whether a given integral self-affine set has non-empty interior. It is worth pointing out that there exist self-affine sets of positive Lebesgue measure which have empty interior; see \cite{CJPPS2006} for such examples in the self-similar setting.

For the convenience of the readers, we illustrate the rough ideas in the proofs of Theorems \ref{thm:general} and \ref{thm:abelian}.  Under the assumptions of
Theorem \ref{thm:general}, we first show that there exist a Borel probability measure $\mu$ on $\Sigma$, $C>0$, $t>d$ and $r\in (0,1)$ such that
$$
\mu([I])\leq C \alpha_d(T_I)^t |\det(T_I)| r^n
$$
for all $n\in \N$ and $I\in \Sigma_n$, where $[I]:=\{x=(x_i)_{i=1}^\infty\in \Sigma\colon x_1\cdots x_n=I\}$; see Lemma \ref{lem:ExistMeasure}. Write $\mu^\ba:=\mu\circ (\pi^\ba)^{-1}$ for $\ba\in \R^{md}$. Clearly $\mu^\ba$ is supported on $K^\ba$.  Let $\widehat{\mu^\ba}$ denote the Fourier transform of $\mu^\ba$; see Section \ref{S-2.1}. We manage to prove that
\begin{equation}
\label{e-key1}
\int_{B(0,\rho)}\int_{\R^d} |\widehat{\mu^\ba}(\xi)|^2\|\xi\|^t\, d\xi d\ba<\infty
\end{equation}
for each $\rho>0$, where $B(0,\rho)$ stands for the closed ball in $\R^{md}$ of radius $\rho$ centred at the origin. The proof of  \eqref{e-key1} is based on some key inequalities (see Propositions \ref{coro:StatPhase} and \ref{lem:t}). By \eqref{e-key1}, for almost all $\ba$,
$\int_{\R^d} |\widehat{\mu^\ba}(\xi)|^2\|\xi\|^t\, d\xi<\infty$; which implies that the Sobolev dimension of $\mu^\ba$ is larger than $2d$, hence $K^\ba$ has non-empty interior (see Definition~\ref{de-2.1} and Lemma \ref{lem-Mat}).  This concludes Theorem \ref{thm:general}.  To prove Theorem \ref{thm:abelian},  our main idea is to construct two self-affine sets $E^\ba, F^\ba\subset \R^d$ for each $\ba\in \R^{md}$ such that $K^\ba$ contains a translation of the sum set $E^\ba+F^\ba$; and moreover for almost all $\ba$, $E^\ba$ and $F^\ba$ have positive Lebesgue measure. By the Steinhaus theorem, $K^\ba$ has non-empty interior for almost all $\ba$. In this approach, the commutative assumption on $T_1,\ldots, T_m$ plays a significant role.

It is worth pointing out that by adapting the proof of Theorem \ref{thm:general} we can give an alternative proof of a known result (see Proposition \ref{lem:JPS-measure}) on the Hausdorff dimension and the absolute continuity of projections of measures under the coding map $\pi^\ba$.  This will be illustrated in Section \ref{S6}.

The paper is organized as follows. In Section \ref{S2} we give some preliminaries.   In Section \ref{S3} we prove two key inequalities that are needed in the proof of Theorem \ref{thm:general}.  The proofs of Theorems \ref{thm:general} and \ref{thm:abelian} are given in Sections \ref{S4} and \ref{S5} respectively. In Section \ref{S6} we give an alternative proof   of Proposition \ref{lem:JPS-measure}.

\section{Preliminaries}
\label{S2}
\subsection{Fourier transform,  Soblev energy and Soblev dimension}
\label{S-2.1}
Recall that  the Fourier transform $\widehat{f}$ of a Lebesgue integrable function $f\in L^1(\R^d)$ is defined by $$\widehat{f}(\xi)=\int_{\R^d} e^{-i\innerprod{\xi}{x}}f(x)\,dx,\qquad \xi\in \R^d,$$
where $\langle\cdot,\cdot\rangle$ stands for the usual inner product in $\R^d$.
Similarly, for a finite measure $\mu$ on $\R^d$ with compact support, the Fourier transform of $\mu$ is defined by $$\widehat{\mu}(\xi)=\int e^{-i\innerprod{\xi}{x}}\,d\mu(x)\qquad \xi\in \R^d.$$

Let $\mathcal S(\R^d)$ denote the Schwartz class of rapidly decreasing functions on $\R^d$.  It consists of infinitely differentiable functions on $\R^d$ all of whose derivatives remain bounded when multiplied by any polynomial.  A basic fact in Fourier analysis is that $f\in \mathcal S(\R^d)$ if and only if $\widehat{f}\in \mathcal S(\R^d)$.
Let $C_0^\infty(\R^d)$ denote the collection of infinitely differentiable functions on $\R^d$ with compact support. Clearly $C_0^\infty(\R^d)\subset \mathcal S(\R^d)$.

Let $\mathcal M(\R^d)$ denote the collection of finite Borel measures on $\R^d$ with compact support.  Following Mattila \cite{Mattila2015} and Peres and Schlag \cite{PeresSchlag2000}, we introduce the following.

\begin{de}
\label{de-2.1}
The Sobolev energy of degree $s\in \R$ of a measure $\mu\in \mathcal M(\R^d)$ is
$$
\mathcal I_s(\mu)=\int_{\R^d} |\widehat{\mu}(x)|^2\|x\|^{s-d}\, dx<\infty,
$$
and the Sobolev dimension of $\mu\in \mathcal M(\R^d)$ is
$$
\dim_S \mu=\sup\left\{s\in \R\colon \mathcal I_s(\mu)<\infty\right\},
$$
where we take the convention that $\sup \emptyset=0$.
\end{de}

The following result is needed in the proof of Theorem \ref{thm:general}.
\begin{lem}[{\cite[Theorem 5.4]{Mattila2015}}]
\label{lem-Mat}
Let $\mu\in \mathcal M(\R^d)$ with $\mu\neq 0$. Suppose that $\dim_S\mu>2d$. Then $\mu$ is absolutely continuous with a continuous density, so its support has non-empty interior.
\end{lem}

\subsection{Singular value function and affinity dimension}

Let ${\rm Mat}_d(\R)$ denote the set of $d\times d$ real matrices. For $A\in {\rm Mat}_d(\R)$, the singular values
$\alpha_1(A)\geq \cdots\geq \alpha_d(A)$ are the square roots of the eigenvalues of $A^*A$. Alternatively, they are the lengths of the semi-axes of the ellipsoid $A(B(0,1))$, where $B(0,1)$ is the unit ball in $\R^d$.

For $s\geq 0$, we define the singular value function $\phi^s\colon {\rm Mat}_d(\R)\to [0,\infty)$ by
\begin{equation}
\label{e-singular}
	\phi^{s}(A)= \begin{cases}
		\alpha_{1}(A)\cdots\alpha_{\lfloor s \rfloor}(A)\alpha_{k}^{s-\lfloor s \rfloor} & \text{ if } 0\leq s\leq d,\\
		|\det (A)|^{s/d} & \text{ if } s > d,
	\end{cases}
\end{equation}
where $\lfloor s \rfloor$ is the integral part of $s$. Here we make the convention $0^0=1$.
\begin{de}\label{def:affinity}
	Let $  (T_{1}, \ldots, T_{m}) $ be a tuple of $d\times d$ real  matrices. The \textit{affinity dimension} of $(T_{1}, \ldots, T_{m})$  is defined by
	\begin{equation*}
	 	\dim_{\rm AFF}(T_{1}, \ldots, T_{m}) = \inf\left \{ s \geq 0 \colon \sum_{n=1}^{\infty}\sum_{I\in\{1,\ldots,m\}^n} \phi^{s}(T_{I}) < \infty \right \}.
	\end{equation*}
\end{de}

\section{Useful inequalities}
\label{S3}
In this section we establish several inequalities (Propositions \ref{coro:StatPhase},  \ref{lem:t} and \ref{lem:t-d*}), of which the first two are needed in the proof of Theorem \ref{thm:general}, and the third one is needed in the proof of Proposition  \ref{lem:JPS-measure}(i).

For $\ba=(a_1,\ldots, a_m)\in \R^{md}$, let $\pi^{\ba}$ be the coding map associated with the IFS
$\{T_ix+a_i\}_{i=1}^m$; see \eqref{e-pia}.
 For a differentiable function $\phi\colon \R^d\to \R$ and $x\in \R^d$, let $\nabla _x\phi$ denote the gradient of $\phi$ at $x$.
We begin with a simple lemma.
\begin{lem}\label{lem-1}
	Let $ \phi:\R^d\to \R $ be a linear function and  $ \psi\in C_0^\infty(\R^d) $. For $ \lambda > 0 $, let
	\begin{equation*}
		I(\lambda) = \int_{\R^d} e^{-i\lambda \phi(x)} \psi(x) \, dx.
	\end{equation*}
Then for each $N\in \N$, there exists $C=C(\psi, N)>0$  such that
	\begin{equation*}
		|I(\lambda)| \leq \frac{C}{\min\{1, \|\nabla \phi\|\}^N}  (1 + \lambda)^{-N} \quad \text{ for all }\; \lambda >0.
	\end{equation*}
\end{lem}

\begin{proof}
Let $N\in \N$. Since $\psi\in C_0^\infty(\R^d)$,  there exists $C=C(\psi, N)>0$ such that
$$
|\widehat{\psi}(\xi)|\leq C(1+\|\xi\|)^{-N}
$$
for $\xi\in \R^d$.
Since $\phi\colon \R^d\to \R$ is linear, there exists $u\in \R^d$ such that  $\phi(x)=\langle u,x\rangle$ for $x\in \R^d$.  Hence  $I(\lambda)=\widehat{\psi}(\lambda u)$ for $\lambda>0$. It follows that for each  $\lambda>0$,
$$
|I(\lambda)|\leq C(1+\|\lambda u\|)^{-N}\leq C\left(1+\lambda\min\{1,\|u\|\}\right)^{-N}\leq \frac{C}{\min\{1, \|u\|\}^N}  (1 + \lambda)^{-N}.
$$
Clearly $\nabla_x\phi=u$ for $x\in \R^d$. This proves the lemma.
\end{proof}

\begin{lem}\label{lem-2}
Assume that $\delta:=\max_{1\leq i\leq m} \|T_{i}\|<1/2 $.  Then
	\begin{equation}\label{key}
		\|{\nabla_{\ba}\innerprod{v}{\pi^{\ba}(x) - \pi^{\ba}(y)}}\| \geq \frac{1-2\delta}{1-\delta}
	\end{equation}
	for all $\ba\in \R^{md}$, $ x = (x_{k})_{k=1}^{\infty}$, $y = (y_{k})_{k=1}^{\infty} \in  \Sigma $ with  $ x_{1} \neq y_{1} $ and $ v \in \R^d $ with $\|v\| =1 $.
\end{lem}
\begin{proof}
The core of our proof follows closely  the proof of \cite[Proposition 3.1]{Solomyak1998}.
Let $ x = (x_{k})_{k=1}^{\infty}$, $y = (y_{k})_{k=1}^{\infty} \in  \Sigma $ with  $ x_{1} \neq y_{1} $,  $ v \in \R^d $ with $\|v\| =1 $ and $\ba=(a_1,\ldots, a_m)\in \R^{md}$.
	Without loss of generality we may assume that $ x_{1} = 1 $ and $ y_{1} = 2 $.
Write ${\bf I}_d:={\rm diag}(\underbrace {1,\ldots, 1}_{d})$. Then
	\begin{align}
\label{e-l1}
		\pi^{\ba}(x) - \pi^{\ba}(y) & = a_{1} - a_{2} + \sum_{k=1}^{\infty} T_{x|k} a_{x_{k+1}} - \sum_{k=1}^{\infty} T_{y|k} a_{y_{k+1}} = \sum_{j=1}^{m} U_{j} a_{j},
  	\end{align}
  where $U_1,\ldots, U_m$ are $d\times d$ matrices defined by
  \begin{equation}
  \label{e-l2}
  \begin{split}
  U_1&={\bf I}_d+\sum_{k\geq 1:\; x_{k+1}=1}T_{x|k}-\sum_{p\geq 1:\; y_{p+1}=1}T_{y|p+1},\\
  U_2&=-{\bf I}_d+\sum_{k\geq 1:\; x_{k+1}=2}T_{x|k}-\sum_{p\geq 1:\; y_{p+1}=2}T_{y|p+1},\\
  U_j&=\sum_{k\geq 1:\; x_{k+1}=j}T_{x|k}-\sum_{p\geq 1:\; y_{p+1}=j}T_{y|p+1} \quad \mbox{ for }3\leq j\leq m.
  \end{split}
  \end{equation}
  Set $A=U_1-{\bf I}_d$ and $B=U_2+{\bf I}_d$.  By \eqref{e-l2},
    \begin{equation*}
  	\norm{A} + \norm{B} + \sum_{j=3}^{m} \norm{U_{j}} \leq   \sum_{k=1}^{\infty} \norm{T_{x|k}}  + \sum_{p=1}^{\infty} \norm{T_{y|p}} \leq 2\sum_{k=1}^\infty \delta^k= \frac{2\delta}{1 - \delta}<2.
  \end{equation*}
   Hence either $ \norm{A} \leq \delta / (1 - \delta) $ or $ \norm{B} \leq \delta / (1 - \delta) $. Suppose that $ \norm{A} \leq \delta / (1 - \delta) $  whilst the other case follows from the same argument. Then
   \begin{equation}
   \label{e-l3}
   	\norm{U_{1}^{-1}} = \norm{({\bf I}_d - A)^{-1}} \leq \sum_{k=0}^{\infty} \norm{A}^{k} \leq \sum_{k=0}^\infty \left(\frac{\delta}{1-\delta}\right)^k= \frac{1 -\delta }{1 - 2\delta}.
   \end{equation}
   By \eqref{e-l1},
   \begin{align*}
   		\nabla_{\ba} \innerprod{v}{\pi^{\ba}(x) - \pi^{\ba}(y)} & = \nabla_{\ba}  \left\langle v,\;\sum_{j=1}^{m} U_{j} a_{j}\right\rangle =  (v^{t}U_{1}, \ldots, v^{t}U_{m}),
   \end{align*}
   where $v^t$ denotes the transpose of $v$.
   Since $ \|v\| = 1 $, it follows that
   $$
   \norm{\nabla_{\ba} \innerprod{v}{\pi^{\ba}(x) - \pi^{\ba}(y)} }  \geq \|v^{t} U_{1}\|\geq \alpha_d(U_1)=
    \norm{U_{1}^{-1}}^{-1} \geq \frac{1 -2 \delta }{1 - \delta},
   $$
	as desired.
\end{proof}

Now we are ready to deduce an important integral estimate. For $x, y\in \Sigma$, let $x\wedge y$ denote the common initial segment of $x$ and $y$.
\begin{pro}\label{coro:StatPhase}
	Assume that $\delta:=\max_{1\leq i\leq m} \|T_{i}\|<1/2 $.   Let $ \psi \in C_{0}^{\infty}(\R^{md}) $ and  $ N \in \Bbb N $. Then there exists $ C=C(\psi, N,\delta) > 0 $ such that
	\begin{equation}
\label{e-l4}
		\left|\int_{\R^{md}} \psi(\ba) e^{-i\innerprod{\xi}{\pi^{\ba}(x) - \pi^{\ba}(y)}} \, d\ba\right|  \leq  C\left(1+ \|T_{x\wedge y}^{\ast} \xi\|\right)^{-N}
	\end{equation}
for all $\xi\in \R^d$ and $x,y\in \Sigma$ with $x\neq y$, where $T_{x\wedge y}^{\ast}$ stands for the transpose of $T_{x\wedge y}$.
\end{pro}
\begin{proof}
	Let $\xi\in \R^d\backslash\{0\}$ and $x,y\in \Sigma$ with $x\neq y$. Let $n=|x\wedge y|$ be the word length of $x\wedge y$. Write $$ v_{\xi, x,y} = \frac{T_{x\wedge y}^{\ast}\xi}{\|T_{x\wedge y}^{\ast}\xi\|}.
$$
Then
\begin{align*}
\innerprod{\xi}{\pi^{\ba}(x) - \pi^{\ba}(y)}&=\innerprod{\xi}{T_{x\wedge y}(\pi^{\ba}(\sigma^n x) - \pi^{\ba}(\sigma^n y))}\\
&=\innerprod{T_{x\wedge y}^{\ast}\xi}{\pi^{\ba}(\sigma^n x) - \pi^{\ba}(\sigma^n y)}\\
&=\|T_{x\wedge y}^{\ast}\xi\| \innerprod{v_{\xi,x,y}}{\pi^{\ba}(\sigma^n x) - \pi^{\ba}(\sigma^n y)}.
\end{align*}
Defining $\phi\colon\R^{md}\to \R$ by $\phi(\ba)=\innerprod{v_{\xi,x,y}}{\pi^{\ba}(\sigma^n x) - \pi^{\ba}(\sigma^n y)}$, we get
	\begin{equation}\label{eq:transferStatPhase}
		\int_{\R^{md}} \psi(\ba) e^{-i\innerprod{\xi}{\pi^{\ba}(x) - \pi^{\ba}(y)}} \, d\ba = \int_{\R^{md}} \psi(\ba) e^{-i\|T_{x\wedge y}^{\ast}\xi\| \phi(\ba)} \, d\ba.
	\end{equation}
Notice that $\phi$ is linear (cf.~\eqref{e-l1}). So by Lemma \ref{lem-2}, $\|\nabla_\ba \phi\|\geq (1-2\delta)/(1-\delta)$. Applying Lemma \ref{lem-1} (in which we take $\lambda=\|T_{x\wedge y}^{\ast}\xi\|$) yields that for each $N\in \N$, there exists $C=C(N, \psi, \delta)>0$ so that
\begin{equation*}
	\left|\int_{\R^{md}} \psi(\ba) e^{-i\|T_{x\wedge y}^{\ast}\xi\| \phi(\ba)}\, d\ba\right|\leq C (1+\|T_{x\wedge y}^{\ast}\xi\|)^{-N}.
\end{equation*}
Combining it with \eqref{eq:transferStatPhase} completes the proof.
\end{proof}

In the remaining part of this section, we shall mean by $ a \lesssim_{\varepsilon} b $ that $ a \leq Cb $ for some positive constant $ C $ depending on $ \varepsilon $. We write $ a \approx_{\varepsilon} b $ if $ a \lesssim_{\varepsilon} b $ and $ b \lesssim_{\varepsilon} a $.


For $d\in \N$, let ${\rm GL}(d, \R)$ denote the collection of $d\times d$ invertible real matrices.
\begin{pro}\label{lem:t}
	Let $ d \in \Bbb N $,   $ t \geq 0 $ and $ N > t + d $. Then
	\begin{equation*}
		\int_{\euclid[d]} (1 + \|Tx\|)^{-N} \|x\|^{t} \, dx \approx_{N,d,t} \frac{1}{\alpha_{d}(T)^{t} \abs{\det (T)}}
	\end{equation*}
for $T\in {\rm GL}(d,\R)$.
\end{pro}
\begin{proof}
	Substituting  $ y = Tx $ gives
	\begin{equation}\label{eq:t-change}
		\int_{\euclid[d]} (1 + \|Tx\|)^{-N} \|x\|^{t}\, dx = \frac{1}{\abs{\det (T)}}  \int_{\euclid[d]} (1 + \|y\|)^{-N} \|T^{-1}y\|^{t} \, dy.
	\end{equation}

	Let $ \beta_{i} = 1/\alpha_{i}(T) $ for $ i = 1, \ldots, d $. Then $ \beta_{1}^{2} , \ldots, \beta_{d}^{2} $  are the eigenvalues of $ (T^{-1})^{\ast}(T^{-1})$.  Choosing coordinate axes in the directions of the eigenvectors of $ (T^{-1})^{\ast}(T^{-1})$ corresponding to $\beta_1^2,\ldots, \beta_d^2$, we obtain
	\begin{equation}\label{eq:t-change'}
\begin{split}
   \int_{\euclid[d]} (1 + \|y\|)^{-N} \|T^{-1}y\|^{t} \, dy&=\int\cdots \int_{\euclid[d]} (1 + \|y\|)^{-N} \left(\beta_{1}^{2}y_{1}^{2}+\cdots+\beta_{d}^{2}y_{d}^{2}\right)^{t/2} \, dy_{1}\cdots dy_{d}\\
    & \approx_{d, t} \int\cdots \int_{\euclid[d]} (1 + \|y\|)^{-N} \left(\beta_{1}^{t}y_{1}^{t}+\cdots+\beta_{d}^{t}y_{d}^{t}\right) \, dy_{1}\cdots dy_{d}\\
    &= \sum_{i=1}^{d}\beta_{i}^{t} c_{i},
\end{split}	
\end{equation}
where $ c_{i}:= \int_{\euclid[d]} (1 + \|y\|)^{-N} \abs{y_{i}}^{t} \, dy_{1}\cdots dy_{d}$.
Since $  c_{i} \approx_{N,d,t} 1 $ by $ N > d + t $, it follows from \eqref{eq:t-change}, \eqref{eq:t-change'} that
	\begin{align*}
		\int_{\euclid[d]} (1 + \|Tx\|)^{-N} \|x\|^{t}\, dx
		& \approx_{d,t} \frac{1}{\abs{\det (T)}}  \sum_{i=1}^{d}\beta_{i}^{t} c_{i} \\
		& \approx_{N, d, t} \frac{1}{\abs{\det (T)}} \max_{i\in \{1,\ldots, d\} } \beta_{i}^{t}\\
& = \frac{1}{\alpha_{d}(T)^{t}\abs{\det (T)}},
	\end{align*}
which completes the proof of the proposition.
\end{proof}

As a complement of Proposition \ref{lem:t}, we have the following.
\begin{pro}\label{lem:t-d*}
	Let  $ d \in \Bbb N $, $ t \in (0, d) \backslash \Bbb Z $ and $ N > t $. Then
	\begin{equation*}
		\int_{\euclid[d]} (1 + \|{Tx}\|)^{-N} \|x\|^{t-d} \, dx \approx_{N,d,t} \frac{1}{\phi^{t}(T)}
	\end{equation*}
	for  $ T\in {\rm GL}(d,\R) $.
\end{pro}

The proof of the above proposition is based a simple lemma.
\begin{lem}\label{lem:ReduceIntegral}
	Let $ d\in \Bbb N $ and $ s > 1 $. Then for $ (x_{1}, \ldots, x_{d}) \in \euclid[d]\backslash \{0\} $,
	\begin{equation*}
		\int_{\Bbb R} \frac{1}{\left(\sum_{i=1}^{d}\abs{x_{i}}^{s}\right) + \abs{y}^{s}} \, dy \approx_{d,s} \frac{1}{\sum_{i=1}^{d}\abs{x_{i}}^{s-1}}.
	\end{equation*}
\end{lem}
\begin{proof}
	Set $ A = \sum_{i=1}^{d}\abs{x_{i}}^{s} $. Then
	\begin{align*}
			\int_{\Bbb R} \frac{1}{\left(\sum_{i=1}^{d}\abs{x_{i}}^{s} \right)+ \abs{y}^{s}} \, dy & = 2 \int_{0}^{\infty} \frac{1}{A + y^{s}} \, dy  \approx \int_{0}^{A^{1/s}} \frac{1}{A} \, dy + 2 \int_{A^{1/s}}^{\infty} \frac{1}{y^{s}} \, dy \\
			& \approx \frac{1}{A^{(s-1)/s}} = \frac{1}{\left(\sum_{i=1}^{d}\abs{x_{i}}^{s}\right)^{(s-1)/s}} \approx_{d, s} \frac{1}{\sum_{i=1}^{d} \abs{x_{i}}^{s-1}},
	\end{align*}
as desired.
\end{proof}

\begin{proof}[Proof of Proposition \ref{lem:t-d*}]
	Suppose $ k < t < k+1 $ for some $ k \in \{ 0, \ldots, d-1\} $. Let $ \alpha_{1}\geq  \cdots\geq  \alpha_{d} $ be the singular values of $ T $. Choosing coordinate axes in the directions of the eigenvectors of $ T^{\ast}T$ corresponding to $\alpha_1^2,\ldots, \alpha_d^2$, we obtain
		\begin{equation}\label{eq:diag}
\begin{split}
		\int_{\euclid[d]} (1 + \|{Tx}\|)^{-N} \|x\|^{t-d} \, dx  &
=\int\cdots \int_{\euclid[d]} \frac{dx_{1}\cdots dx_{d}}{\left(1 + \sqrt{\sum_{i=1}^{d}\abs{\alpha_{i} x_{i}}^{2}}\right)^{N}\left(\sum_{j=1}^{d}\abs{x_{j}}^{2}\right)^{(d-t)/2}}\\
 &\approx_{N, d, t} \int\cdots \int_{\euclid[d]} \frac{dx_{1}\cdots dx_{d}}{\left(1 + \sum_{i=1}^{d}\abs{\alpha_{i} x_{i}}^{N}\right)\left(\sum_{j=1}^{d}\abs{x_{j}}^{d-t}\right)}.
\end{split}
	\end{equation}
	  Since $ d - t > d - (k+1) $, applying Lemma \ref{lem:ReduceIntegral} repeatedly yields
	 \begin{equation}\label{eq:AppReduce}
	 	\int\cdots\int_{\euclid[d - (k+1)]} \frac{1}{\sum_{i=1}^{d}\abs{x_{i}}^{d-t}} \, dx_{k+2}\cdots dx_{d} \approx_{d,t} \frac{1}{\sum_{i=1}^{k+1}\abs{x_{i}}^{k+1 - t}}.
	 \end{equation}
	  We make a convention that $ \alpha_{1}\cdots \alpha_{k} = 1 $ if $ k = 0 $. Then
	 \begin{equation}\label{eq:1stReduce}
	 	\begin{aligned}
	 	 \int_{\euclid[d]} & (1 + \|{Tx}\|)^{-N} \|{x}\|^{t-d} \, dx\\
	 	& \lesssim_{N,d,t}\int\cdots \int_{\euclid[d]} \frac{dx_{1}\cdots dx_{d}}{\left(1 + \sum_{i=1}^{k+1}\abs{\alpha_{i} x_{i}}^{N}\right)\left(\sum_{j=1}^{d}\abs{x_{j}}^{d-t}\right)}  & (\text{by } \eqref{eq:diag})  \\
 	 	&  \approx_{N,d,t} \int\cdots \int_{\euclid[k+1]} \frac{dx_{1}\cdots dx_{k+1}}{\left(1 + \sum_{i=1}^{k+1}\abs{\alpha_{i} x_{i}}^{N}\right)\left(\sum_{j=1}^{k+1}\abs{x_{j}}^{k+1-t}\right)} & (\text{by } \eqref{eq:AppReduce}) \\
 &  \leq \int\cdots \int_{\euclid[k+1]} \frac{dx_{1}\cdots dx_{k+1}}{\left(1 + \sum_{i=1}^{k+1}\abs{\alpha_{i} x_{i}}^{N}\right) |x_{k+1}|^{k+1-t}}  \\
	 	& = \frac{1}{\alpha_{1}\cdots \alpha_{k} \alpha_{k+1}^{t-k}} \int\cdots \int_{\euclid[k+1]} \frac{dy_{1}\cdots dy_{k+1}}{\left(1 + \sum_{i=1}^{k+1}\abs{y_{i}}^{N} \right) \abs{y_{k+1}}^{k+1 - t}}\\
	 	& = \frac{1}{\phi^{t}(T)} \int\cdots \int_{\euclid[k+1]} \frac{dy_{1}\cdots dy_{k+1}}{\left(1 + \sum_{i=1}^{k+1}\abs{y_{i}}^{N} \right) \abs{y_{k+1}}^{k+1 - t}}
	 \end{aligned}
	 \end{equation}
 	where in the second last equality we took a change of variables via $y_i=\alpha_ix$ for $1\leq i\leq k+1$.
 	 Since $ N > t > k $, applying Lemma \ref{lem:ReduceIntegral} repeatedly to the variables $ y_{1}, \ldots, y_{k} $ yields
 	 \begin{equation}\label{eq:t-d-UB}
 	 	\begin{aligned}
 	 	\int\cdots \int_{\euclid[k+1]} \frac{dy_{1}\cdots dy_{k+1}}{\left(1 + \sum_{i=1}^{k+1}\abs{y_{i}}^{N} \right) \abs{y_{k+1}}^{k+1 - t}}  & \lesssim_{N,d,t} \int_{\Bbb R}   \frac{1}{\left(1 + \abs{ y_{k+1}}^{N - k}\right)\abs{y_{k+1}}^{k+1 - t}} \, dy_{k+1} \\
 	 	& \lesssim_{N, d, t} 1
 	 \end{aligned}
 	 \end{equation}
  where in the last inequality we used $ N - k + (k+1) - t = N -t + 1 > 1 $ and $ 0 < k+1-t < 1 $. Combining
\eqref{eq:1stReduce} with \eqref{eq:t-d-UB} yields the upper bound
$$
\int_{\euclid[d]} (1 + \|{Tx}\|)^{-N} \|x\|^{t-d} \, dx \lesssim_{N,d,t}\frac{1}{\phi^{t}(T)}.
$$

Next we prove the lower bound. Let $\Omega$ denote the set of $(x_{1}, \ldots, x_{d}) \in \euclid[d]$ such that
\begin{align*}
&\abs{x_{i}} \leq \frac{1}{\alpha_{i}}\quad  \mbox{  for }1\leq i\leq k,\\
& \frac{1}{\alpha_{k+1}} \leq   \abs{x_{k+1}} \leq  \frac{2}{\alpha_{k+1}} \quad \mbox{  and }\\
& \abs{x_{j}} \leq   \frac{1}{\alpha_{k+1}}\quad  \mbox{  for } k+2\leq j\leq d.
\end{align*}
Then for each $ (x_{1}, \ldots, x_{d}) \in \Omega $,  we have \begin{equation*}
 	\sum_{i=1}^{d} \abs{x_{i}}^{d-t}  \leq 2^{d-t}d \alpha_{k+1}^{t-d} \lesssim_{d, t} \alpha_{k+1}^{t-d}
 \end{equation*}
	and $ \abs{\alpha_{i}x_{i}} \leq 2 $ for $1\leq i\leq d $ which implies
 \begin{equation*}
 	 \frac{1}{1 + \sum_{i=1}^{d}\abs{\alpha_{i} x_{i}}^{N}}  \geq \frac{1}{1 + 2^{N}d} \succsim_{N,d} 1.
 \end{equation*}
	Hence by \eqref{eq:diag},
  \begin{equation*}\label{eq:t-d-LB}
  	\begin{aligned}
  	\int_{\euclid[d]} (1 + \|Tx\|)^{-N} \|x\|^{t-d} \, dx  & \succsim_{N, d, t} \int\cdots \int_{\Omega} \frac{dx_{1}\cdots dx_{d}}{\left(1 + \sum_{i=1}^{d}\abs{\alpha_{i} x_{i}}^{N}\right)\left( \sum_{j=1}^{d}\abs{x_{j}}^{d-t}\right)} \, \\
  	& \succsim_{N, d, t}  \int\cdots \int_{\Omega} \alpha_{k+1}^{d-t} \, dx \\
  	& = \mathcal L^{d}(\Omega)  \alpha_{k+1}^{d-t} \\
  	& \succsim_{d} \frac{1}{\alpha_{1}\cdots\alpha_{k} \alpha_{k+1}^{d-k} \alpha_{k+1}^{t - d}} = \frac{1}{\phi^{t}(T)}.
  \end{aligned}
  \end{equation*}
 This finishes the proof of the lower bound.
\end{proof}

\section{The proofs of Theorem \ref{thm:general} and Corollary \ref{cor-1.2}}
\label{S4}

We first introduce some notation.  For $n\in \N$ write $\Sigma_n=\{1,\ldots, m\}^n$. Set $\Sigma_0=\{\varnothing\}$ where $\varnothing$ stands for the empty word. Write $\Sigma^*=\bigcup_{n=0}^\infty \Sigma_n$.  Set $|I|=n$ for every $I\in \Sigma_n$. For  each $t\geq 0$, we define $g_t:\Sigma^*\to (0,\infty)$ by
\begin{equation}
\label{e-gt*}
g_t(I)=\alpha_d(T_I)^t|\det(T_I)|,
\end{equation}
where we take the convention that $g_t(\varnothing)=1$.
\begin{lem}
Let $t\geq 0$. Then $g_t$ is super-multiplicative on $\Sigma^*$ in the sense that
\begin{equation}
\label{e-e0}
g_t(IJ)\geq g_t(I)g_t(J)\quad \mbox{  for all }\, I, J\in \Sigma^*.
\end{equation}
Consequently,
\begin{equation}\label{eq:limSupAdd}
	\lim_{n\to\infty} \left (\sum_{I\in \Sigma_{n}}g_{t}(I) \right )^{1/n} = \sup_{n\in\Bbb N} \left (\sum_{I\in \Sigma_{n}}g_{t}(I) \right )^{1/n}.
\end{equation}
\end{lem}
\begin{proof}
Notice that for $ I, J \in \Sigma^{\ast} $,
$$
\alpha_d(T_{IJ})= \norm{T_{IJ}^{-1}}^{-1} \geq \norm{T_{I}^{-1}}^{-1}\norm{T_{J}^{-1}}^{-1}=\alpha_d(T_{I})\alpha_d(T_{J}).
$$
It follows that
\begin{equation*}
	g_{t}(IJ) = \alpha_{d}(T_{IJ})^t \abs{\det (T_{IJ})} \geq \alpha_d(T_{I})^t\alpha_d(T_{J})^t \abs{\det (T_{I})}\abs{\det (T_{J})} =g_t(I)g_t(J).
\end{equation*}
Hence \eqref{e-e0} holds.
Set $a_n:= \sum_{I\in \Sigma_{n}}g_{t}(I)$ for $n\in \N$.  By \eqref{e-e0}, $a_{n+m}\geq a_na_m$ for all $n,m\in \N$, from which \eqref{eq:limSupAdd} follows.
\end{proof}

The next lemma allows us to construct a certain regular measure on $\Sigma$ under the assumptions of Theorem \ref{thm:general}.
\begin{lem}\label{lem:ExistMeasure}
	Suppose  $t(T_{1}, \ldots, T_{m}) > d$. Then for every $t$ with  $ d<t<t(T_1,\ldots, T_m)$, there exist a Borel probability measure  $\mu$ on $\Sigma$,  $r\in (0,1)$ and $C>0$ such that
	\begin{equation}
	\label{e-e2}
		\mu([I]) \leq  C g_{t}(I) r^{\abs{I}} \quad \mbox{ for all }\, I \in \Sigma^{\ast},
	\end{equation}
where $[I]:=\{x=(x_n)_{n=1}^\infty\in \Sigma:\; x_1\cdots x_k=I\}$ for $I\in \Sigma_k$.
	\end{lem}
\begin{proof}
Let $ d<t<t(T_1,\ldots, T_m)$.  By the definition of $ t(T_1,\ldots, T_m) $, there exists $ N \in \Bbb N $ such that
	\begin{equation*}
		\lambda:=\sum_{I\in \Sigma_{N}}g_{t}(I) >1.
	\end{equation*}
  Define a probability vector $p=\{p_I\}_{I\in \Sigma_N}$  by
 $$
p_I = g_{t}(I)\lambda^{-1},\quad I\in \Sigma_{N}.
$$
Let $\mu$ be the Bernoulli product measure on $\Sigma=(\Sigma_N)^\N$ associated with $p$. That is,
	\begin{equation} \label{eq:prodMeasure}
		\mu([I_{1}\ldots I_{k}]) = \prod_{\ell=1}^{k}p_{I_{\ell}}
	\end{equation}
	for every $k\in \N$ and $ I_{1},\ldots I_{k} \in \Sigma_{N}$.
	
Set $r=\lambda^{-1/N}$. Then $0<r<1$. 	Next we show that \eqref{e-e2} holds for some $C>0$. To this end, let $ I \in \Sigma^{\ast} $.
Then $I$ can be written as  $ I = I_{1}\ldots I_{k} W $  with $I_1,\ldots, I_k\in \Sigma_N$ and  $ 1\leq \abs{W}\leq N $. It may happen that $k=0$ and in that case $I_{1}\ldots I_{k}$ should be viewed as the empty word $\varnothing$.  It follows from \eqref{e-e0} that
	\begin{equation}\label{eq:transBack}
		\frac{g_{t}(I_{1}\ldots I_{k}) }{g_{t}(I)}  = \frac{g_{t}(I_{1}\ldots I_{k}) }{g_{t}( I_{1}\ldots I_{k} W )} \leq \dfrac{1}{g_{t}(W)} \leq \gamma
	\end{equation}
where $\gamma: = \max\{ 1/g_{t}(J) \colon \abs{J} \leq N \} < \infty $. By \eqref{eq:prodMeasure}, \eqref{e-e0} and \eqref{eq:transBack},
	\begin{align*}
	\mu([I])  \leq \mu([I_{1} \ldots I_{k}])
	= \prod_{\ell=1}^{k} \dfrac{g_{t}(I_{\ell})}{\lambda}
 \leq g_{t}(I_{1}\ldots I_{k})\lambda^{-k}
 \leq \gamma  g_{t}(I) \lambda^{-k}.
 \end{align*}
Letting $ C = \gamma \lambda  $ and using $  k \geq \abs{I}/N - 1 $, we  see that $\gamma \lambda^{-k}\leq C (\lambda^{-1/N})^{|I|}=C r^{|I|}$, so
$\mu(I)\leq C g_{t}(I) r^{|I|}$.
\end{proof}

Now we are ready to prove Theorem \ref{thm:general}.

\begin{proof}[Proof of Theorem \ref{thm:general}]
	Fix $t$ so that  $d<t<t(T_1,\ldots, T_m)$.  By Lemma \ref{lem:ExistMeasure}, there exist a Borel probability measure $\mu$ on $\Sigma$, $r\in (0,1)$ and $C>0$  such that
	\begin{equation}
	\label{e-e2*}
		\mu([I]) \leq  C g_{t}(I) r^{\abs{I}} \quad \mbox{ for all }\, I \in \Sigma^{\ast} .
	\end{equation}
Notice that $\mu([I])\to 0$ as $|I|\to \infty$. So $\mu$ has no atoms.
	
	For brevity we write $\mu^\ba=\mu\circ (\pi^\ba)^{-1}$ for $\ba\in \R^{md}$.  Clearly, $\mu^\ba$ is supported on $K^\ba$ for each $\ba$.
For $\rho>0$, let $B(0,\rho)$ denote the closed ball in $\R^{md}$ of radius $\rho$ centred at the origin. 	We claim that for each $\rho>0$,
\begin{equation}\label{eq:generalGoal}
		\int_{B(0,\rho)} \int_{\euclid[d]} \abs{\widehat{\mu^{\ba}}(\xi)}^{2} \|{\xi}\|^{t} \, d\xi\, d\ba < \infty.
	\end{equation}
Clearly, \eqref{eq:generalGoal} implies that for $\mathcal L^{md}$-a.e.~$\ba\in B(0,\rho)$, 	$$\int_{\euclid[d]} \abs{\widehat{\mu^{\ba}}(\xi)}^{2} \|{\xi}\|^{t} \, d\xi<\infty;$$  so $\dim_S \mu^\ba\geq t+d>2d$ by Definition \ref{de-2.1}.
	By Lemma \ref{lem-Mat},  $K^\ba$ has non-empty interior for $\mathcal L^{md}$-a.e.~$\ba\in B(0,\rho)$.
		
In what follows we prove \eqref{eq:generalGoal}.  Fix $\rho>0$.
	Take $\psi \in C_{0}^{\infty}(\euclid[md]) $ such that  $0\leq \psi\leq 1$ and  $ \psi(x)= 1 $ for all $x\in B(0,\rho)$. Applying  Fubini's theorem,
	\begin{equation}
	\label{eq:prepare}
		\begin{aligned}
		 \int_{B(0,\rho)} & \int_{\euclid[d]} \abs{\widehat{\mu^{\ba}}(\xi)}^{2} \|{\xi}\|^{t} \, d\xi\, d\ba \\
		& \leq  \int_{\euclid[md]} \int_{\euclid[d]}   \psi(\ba) \abs{\widehat{\mu^{\ba}}(\xi)}^{2} \|{\xi}\|^{t}  \, d\ba \, d\xi \\
		& =    \int_{\euclid[md]} \int_{\euclid[d]}   \psi(\ba)\|{\xi}\|^{t}  \int_{\Sigma} \int_{\Sigma} e^{-i\innerprod{\xi}{\pi^{\ba}(x) - \pi^{\ba}(y)}} \,d\mu(x)d\mu(y)  \, d\ba \, d\xi \\
		&  = \int_{\Sigma} \int_{\Sigma}  \int_{\euclid[d]} \|{\xi}\|^{t}   \int_{\euclid[md]} \psi(\ba) e^{-i\innerprod{\xi}{\pi^{\ba}(x) - \pi^{\ba}(y)}} \, d\ba \, d\xi  \,d\mu(x)d\mu(y).
	\end{aligned}
	\end{equation}
	
Take  $N>t+d$.  By Proposition  \ref{coro:StatPhase}, there exists $\widetilde{C}>0$ such that
\begin{equation}\label{eq:AppStatPhase}
		\left|\int_{\euclid[md]} \psi(\ba) e^{-i\innerprod{\xi}{\pi^{\ba}(x) - \pi^{\ba}(y)}} \, d\ba\right| \leq \widetilde{C} (1+\|{T_{x\wedge y}^{\ast} \xi}\|)^{-N}  	\end{equation}
 for all $\xi\in \R^d$ and $x,y\in \Sigma$  with $x\neq y$.  Since $\mu$ has no atoms, $\mu\times \mu$ is fully supported on $\{(x,y)\in \Sigma\times \Sigma:\; x\neq y\}$. To see this, simply notice that
 $$
 \mu\times \mu\{(x,x):\; x\in \Sigma\}\leq \sum_{I\in \Sigma_n} \mu([I])^2\leq \sup_{I\in \Sigma_n} \mu([I])\to 0 \;\mbox { as }n\to \infty.
 $$
Hence, by  \eqref{eq:prepare} and \eqref{eq:AppStatPhase},
\begin{equation*}
	\label{eq:prepare*}
		\begin{aligned}
		 \int_{B(0,\rho)} & \int_{\euclid[d]} \abs{\widehat{\mu^{\ba}}(\xi)}^{2} \|{\xi}\|^{t} \, d\xi\, d\ba \\
		 &\leq \widetilde{C}\int_{\Sigma} \int_{\Sigma}  \int_{\euclid[d]} (1+\|{T_{x\wedge y}^{\ast} \xi}\|)^{-N} \|{\xi}\|^{t} \, d\xi  \,d\mu(x)d\mu(y) \\
		& \leq C'  \int_{\Sigma} \int_{\Sigma}  g_{t}(x\wedge y)^{-1}  \,d\mu(x)d\mu(y)  & (\text{by Proposition \ref{lem:t}}) \\
 		& \leq C' \sum_{n=0}^{\infty} \sum_{I\in\Sigma_{n}} g_{t}(I)^{-1} \mu([I])^2\\
 		& \leq C''  \sum_{n=0}^{\infty} \sum_{I\in\Sigma_{n}} r^{n} \mu([I]) & (\text{by } \eqref{e-e2}) \\
 		& = \frac{C''}{1-r} < \infty,
	\end{aligned}
	\end{equation*}
where $C', C''$ are two positive constants.  This proves \eqref{eq:generalGoal}.
\end{proof}

\begin{proof}[Proof of Corollary \ref{cor-1.2}]  Clearly the condition  $\sum_{i=1}^m \alpha_d(T_i)^d |\det(T_i)|>1$ implies that $t(T_1,\ldots, T_m)> d$.   Hence by Theorem \ref{thm:general}, $K^\ba$ has non-empty interior for $\mathcal L^{md}$-a.e.~$\ba\in \R^{md}$ if condition (i) holds.

Notice that whenever $T_i$ ($i=1,\ldots, m$) are scalar multiples of orthogonal matrices,
$$\sum_{i=1}^m \alpha_d(T_i)^d|\det(T_i)|=\sum_{i=1}^m |\det(T_i)|^2.$$
 Hence condition (ii) implies condition (i).
\end{proof}

\section{The proof of Theorem \ref{thm:abelian}}
\label{S5}

In this section we prove Theorem \ref{thm:abelian}. For $\ba=(a_1,\ldots, a_m)\in \R^{md}$, let $\pi^{\ba}$ be the coding map associated with the IFS
$\{T_ix+a_i\}_{i=1}^m$; see \eqref{e-pia}.  Let $\Sigma^*$ be defined as in the beginning part of Section \ref{S4}. Recall that for a Borel probability measure $\eta$ on $\R^d$, its  Hausdorff
dimension $\dim_{\rm H}\eta$  is the smallest Hausdorff dimension of a Borel set $F$  of positive $\eta$ measure.
Part (ii) of the following result is needed in our proof.

\begin{pro}{\cite[Proposition~4.4]{JordanEtAl2007}}
\label{lem:JPS-measure}
	Assume that $\|T_i\|<1/2$ for $1\leq i\leq m$ . Let $ \mu $ be a Borel probability measure on $ \Sigma $. Suppose that there exist $ s >0 $ and $ C > 0 $ such that
	\begin{equation*}
		\mu([I]) \leq C \phi^{s}(T_{I})  \quad\text{ for } I \in \Sigma^{\ast},
	\end{equation*}
where $\phi^s$ is the singular value function defined as in \eqref{e-singular}. 	
	Then the following properties hold:
\begin{itemize}
\item[(i)] If $0<s\leq d$, then $\dim_{\rm H}  \mu^\ba\geq s$ for $ \mathcal L^{md} $-a.e.\ $ \ba \in \euclid[md] $, where $\mu^\ba:=\mu\circ (\pi^\ba)^{-1}$.
\item[(ii)] If $s>d$, then $ \mu^\ba \ll \mathcal L^{d} $ for $ \mathcal L^{md} $-a.e.\ $ \ba \in \euclid[md] $.
\end{itemize}
\end{pro}

We remark that part (i) of the above proposition was also implicitly proved in \cite{Falconer1988}.  In Section \ref{S6}, we will provide an alternative proof of Proposition \ref{lem:JPS-measure} by adapting the proof of Theorem \ref{thm:general}.

\begin{proof}[Proof of Theorem \ref{thm:abelian}]  The proof is conducted as follows. For each $\ba\in \R^{md}$, we  will construct two compact sets $E^{\ba}, F^{\ba}\subset \R^d$ and a vector $v^{\ba}\in \R^d$ such that
$$
K^{\ba}\supset E^{\ba}+F^{\ba}+v^{\ba}:=\{x+y+v^{\ba}:\; x\in E^\ba,\, y\in F^\ba\}.
$$
Then we will show that  both $E^{\ba}$ and $F^\ba$ have positive $d$-dimensional Lebesgue measure for $\mathcal L^{md}$-a.e.~$\ba\in \R^{md}$.
	Clearly by the Steinhaus theorem (see e.g.~\cite{Stromberg1972}),  $K^\ba$ has nonempty interior for $\mathcal L^{md}$-a.e.~$\ba\in \R^{md}$.

	Before giving our constructions of $E^\ba$, $F^\ba$ and $v^\ba$ for $\ba\in \R^{md}$, we first make some preparation.  Write
	$$\mathcal T_{n}:= \{ T_{I} \colon I \in \Sigma_{n} \},\quad n\in \N.$$	
	 Since $ T_{1}, \ldots, T_{m} $ commute, each element in  $\mathcal T_{n}$ is of the form $T_1^{p_1}\cdots T_m^{p_m}$, with $p_1,\ldots, p_m$ being nonnegative integers so that $p_1+\cdots+ p_m=n$.  It follows that
	\begin{equation}
	\label{e-o1}
		\# \mathcal T_{n} \leq (n+1)^{m} \quad \mbox{ for every }n\in \N,
	\end{equation}
 where $\#$ stands for cardinality. 	

 	Since $ \sum_{i=1}^{m} \abs{\det (T_{i})}^{2} > 1 $, by continuity we can choose  $ t > 2 $ such that $$\lambda:= \sum_{i=1}^{m} \abs{\det (T_{i})}^{t} > 1.$$
  Then for $ n \in \Bbb N $,
	\begin{equation}
	\label{e-o2}
	\begin{aligned}
	 	\lambda^{n} &=  \sum_{I\in \Sigma_{n}} \abs{\det (T_{I})}^{t} = \sum_{A \in \mathcal T_{n}} \sum_{I \in \Sigma_{n} \colon T_{I} = A }  \abs{\det (T_{I}) }^{t} \\
		&= \sum_{A \in \mathcal T_{n}} \# \{ I \in \Sigma_{n} \colon T_{I} = A \} \cdot \abs{\det (A) }^{t} .
		\end{aligned}
	\end{equation}
	Since $\lambda>1$, we can choose a large  positive integer $ N$  such that $\lambda^{N}>(N+1)^m$. Then by \eqref{e-o1}, $$ \lambda^{N}> \# \mathcal T_{N}.$$  Applying this to \eqref{e-o2} (in which we take $n=N$) yields that  there exists $ A \in \mathcal T_{N} $ such that
	\begin{equation*}\label{eq:sumDet>1}
		\# \{ I \in \Sigma_{N} \colon T_{I} = A \} \cdot \abs{\det (A) }^{t} > 1.
	\end{equation*}
	Setting $ \mathcal A = \{ I \in \Sigma_{N} \colon T_{I} = A \}  $, we obtain
\begin{equation}
\label{e-o3}
	  (\#  \mathcal A)\cdot  \abs{\det A}^{t} > 1.
	 \end{equation}
	Fix an element $ J \in  \mathcal A $.
	
	Let $\ba=(a_1,\ldots, a_m)\in \R^{md}$. For $ I = i_{1} \ldots i_{N} \in \Sigma_{N} $, define $ 	a_{I} = \sum_{k=0}^{N-1} T_{i_1\ldots i_k}a_{i_{k+1}} $. Then it is easily checked that
	$$f_I^\ba(x):=f^\ba_{i_1}\circ\cdots\circ f^{\ba}_{i_N}(x)=T_Ix+a_I,\quad I\in \Sigma_N.$$
Hence by the definition of $\mathcal A$,  $f_I^\ba(x)=Ax+a_I$ for each $I\in \mathcal A$.	It follows that $\{ Ax + a_{I} \}_{I\in  \mathcal A}$ is a sub-family of the IFS $\{f^\ba_I\}_{I\in \Sigma_N}$. 	Therefore,  letting  $ G^{\ba} $ be the attractor of $ \{ Ax + a_{I} \}_{I\in  \mathcal A} $, we have
	\begin{equation}
\label{e-o4}
		K^{\ba} \supset G^{\ba} =\left  \{ \sum_{k=0}^{\infty} A^{k} a_{I_{k+1}} \colon (I_{k})_{k=1}^{\infty} \in  \mathcal A^{\Bbb N} \right \}.
	\end{equation}

 Notice that for  each $  (I_{k})_{k=1}^{\infty} \in  \mathcal A^{\Bbb N} $,
	\begin{equation}
	\label{e-o3'}
	\begin{aligned}
		 \quad\; \sum_{k=0}^{\infty} A^{k} a_{I_{k+1}}
		 & = \sum_{k=0}^{\infty} \left(A^{2k}a_{I_{2k+1}} + A^{2k+1}a_{I_{2k+2}}\right) \\
		& = \left ( \sum_{k=0}^{\infty} \left(A^{2k}a_{I_{2k+1}}+ A^{2k+1}a_{J}\right) \right ) + \\
		& \; \quad  \left (\sum_{k=0}^{\infty} \left(A^{2k+1}a_{I_{2k+2}} +   A^{2k}a_{J}\right)\right )
		- \sum_{k=0}^{\infty} A^{k}a_{J}.
	\end{aligned}
	\end{equation}
(Recall that $J$ is a fixed element in $\mathcal A$.)
Define $ v^\ba= - \sum_{k=0}^{\infty} A^{k}a_{J} $ and
\begin{align*}
	E^{\ba} &= \left \{ \sum_{k=0}^{\infty} A^{k} a_{I_{k+1}} \colon I_{2n+1} \in  \mathcal A \text{ and } I_{2n+2} =J \text{ for all } n \geq 0 \right \},\\
	F^{\ba} & = \left \{ \sum_{k=0}^{\infty} A^{k} a_{I_{k+1}} \colon I_{2n+1} = J \text{ and } I_{2n+2} \in  \mathcal A   \text{ for all } n \geq 0 \right \}.
\end{align*}
By \eqref{e-o4} and \eqref{e-o3'},
\begin{equation*}
	K^{\ba} \supset G^\ba = E^\ba+ F^\ba + v^\ba.
\end{equation*}

Next we show that for $\mathcal L^{md}$-a.e.~$\ba\in \R^{md}$, $\mathcal L^{d}(E^\ba)>0$ and $\mathcal L^{d}(F^\ba)>0$.
Noticing that $ F^{\ba} = A E^{\ba}+a_J $ with $ A $ being invertible,  so we only need to show that $\mathcal L^d(E^\ba)>0$ for $\mathcal L^{md}$-a.e.~$\ba$.

Define $ \Lambda=\{IJ:\; I\in \mathcal A\}$.   Then $ \Lambda $ is a subset of $ \Sigma_{2N} $,  so $\Lambda^\N$ is a compact subset of $\Sigma$ since $\Sigma=(\Sigma_{2N})^\N$.  By the definition of $E^\ba$, we see that $E^\ba=\pi^\ba(\Lambda^\N)$. Let $ \mu $ be the Bernoulli product measure on $\Lambda^\N$ associated the uniform probability vector $(1/\#  \mathcal A,\ldots, 1/\#  \mathcal A)$. That is,
\begin{equation}
\label{e-o5}
\mu([\omega_1\cdots\omega_n])=\left(\frac{1}{\#  \mathcal A}\right)^n \quad\mbox{ for all } n\in \N \mbox{ and } \omega_1,\ldots,\omega_n\in \Lambda.
\end{equation}
Since $\Lambda^\N$ is a compact subset of $\Sigma$, $\mu$ can be viewed as a Borel probability measure on $\Sigma$. In particular, $\pi^\ba_*\mu=\mu\circ (\pi^\ba)^{-1}$ is supported on $E^\ba$ for each $\ba\in \R^{md}$.

Now we claim that there exists $C>0$ such that
\begin{equation}
\label{e-o7}
\mu([I])\leq C \phi^{t/2}(T_I)\quad \mbox{ for all }I\in \Sigma^*,
\end{equation}
where $\phi^s$ denotes the singular value function defined as in \eqref{e-singular}.  To prove the claim, let $I\in \Sigma^*$. Then there is a unique integer $k\geq 0$ such that $2kN\leq |I|< 2(k+1)N$. Write $I=I_1I_2$ with $|I_1|=2kN$. Clearly, $\mu([I])\leq \mu([I_1])$. If $I_1\not\in \Lambda^k$, then $\mu([I_1])=0$ since $\mu$ is supported on $\Lambda^\N$. Otherwise if $I_1\in \Lambda^k$,  by \eqref{e-o5} and \eqref{e-o3},
\begin{equation}
\label{e-o6}
\mu([I])\leq \mu([I_1])=\left(\frac{1}{\#  \mathcal A}\right)^k<|\det(A)|^{kt}=|\det(T_{I_1})|^{t/2},
\end{equation}
where in the last equality we have used the fact that $I_1\in \mathcal A^{2k}$ which implies $T_{I_1}=A^{2k}$. Since $$|\det(T_I)|=|\det(T_{I_1})\det(T_{I_2})|\geq  |\det(T_{I_1})| \left(\min_{1\leq i\leq m} |\det(T_i)|\right)^{2N},$$
it follows from \eqref{e-o6} that
$$
\mu([I])\leq |\det(T_{I_1})|^{t/2}\leq C |\det(T_I)|^{t/2}=C \phi^{dt/2}(T_I),
$$
where $C:=(\min_{1\leq i\leq m}|\det(T_i)|)^{-tN}$, and in the last equality we have used that $t>2$. This completes the proof of \eqref{e-o7}.

Finally by \eqref{e-o7} and Proposition \ref{lem:JPS-measure}(ii), $\mu^\ba\ll \mathcal L^d$ for $\mathcal L^{md}$-a.e.~$\ba\in \R^{md}$, where $\mu^\ba:=\mu\circ (\pi^\ba)^{-1}$.    Since $ \mu^{\ba}$ is supported  on $ E^\ba $ for each $\ba$, it follows that $\mathcal L^{d}(E^\ba)>0$ for $ \mathcal L^{md} $-a.e. $ \ba \in \euclid[md] $.
\end{proof}

\section{An alternative proof of Proposition  \ref{lem:JPS-measure}}
\label{S6}

We remark that Proposition  \ref{lem:JPS-measure} can be alternatively proved by estimating the Sobolev energies and Sobolev dimension of $\mu^\ba$. Below we give a sketched proof.

\begin{proof}[Sketched proof of Proposition \ref{lem:JPS-measure}]
We first prove (i). Take any  $t\in (0,s)\backslash \Z$. Then there exists $r\in (0,1)$ such that
\begin{equation}
\label{e-ttt}\mu([I]) \leq C r^{|I|} \phi^t(T_I)   \quad\text{ for } I \in \Sigma^{\ast}.
\end{equation}
Take $N>t$. For each $\rho>0$, following the   proof of Theorem \ref{thm:general} with minor changes,  we obtain
\begin{equation*}
\begin{aligned} \int_{B(0,\rho)}&  \int_{\euclid[d]} \abs{\widehat{\mu^{\ba}}(\xi)}^{2} \|\xi\|^{t-d}\, d\xi\, d\ba\\
		 &\leq \widetilde{C}\int_{\Sigma} \int_{\Sigma}  \int_{\euclid[d]} (1+\|{T_{x\wedge y}^{\ast} \xi}\|)^{-N} \|{\xi}\|^{t-d} \, d\xi  \,d\mu(x)d\mu(y) \\
		& \leq C'  \int_{\Sigma} \int_{\Sigma} \left( \phi^{t}(T_{x\wedge y})\right)^{-1}  \,d\mu(x)d\mu(y)  & (\text{by Proposition \ref{lem:t-d*}}) \\
 		& \leq C' \sum_{n=0}^{\infty} \sum_{I\in\Sigma_{n}} \left(\phi^{t}(T_I)\right)^{-1} \mu([I])^2\\
 		& \leq C''  \sum_{n=0}^{\infty} \sum_{I\in\Sigma_{n}} r^{n} \mu([I])= \frac{C''}{1-r} < \infty, & (\text{by } \eqref{e-ttt}) \\
 			\end{aligned}
	\end{equation*}
where $\widetilde{C}, C', C''$ are  positive constants. This implies that $\mathcal I_t(\mu^\ba)<\infty$ and hence  $\dim_S\mu^\ba\geq t$ for almost all $\ba$. It is known that $\dim_{\rm H}\eta\geq \min\{\dim_S\eta,\, d\}$ for each Borel probability measure $\eta$ on $\R^d$ (see e.g.~\cite[p.~199]{PeresSchlag2000}). Hence $\dim_{\rm H}\mu^\ba\geq t$ for almost all $\ba$. Since $t$ is arbitrarily taken from $(0,s)\backslash \Z$, we obtain $\dim_{\rm H}\mu^\ba\geq s$ for almost all $\ba$.

To prove (ii), notice that there is $r\in (0,1)$ such that 
\begin{equation}
\label{e-ttt*}
\mu([I]) \leq C r^{|I|}|\det(T_I)|= C r^{|I|}g_0(T_I)  \quad\text{ for } I \in \Sigma^{\ast},
\end{equation}
where $g_0$ is defined as in \eqref{e-gt*}. Take $N>d$.  For each $\rho>0$, following the proof of Theorem \ref{thm:general} (in which take $t=0$) yields
\begin{equation*}
\begin{aligned} \int_{B(0,\rho)}&  \int_{\euclid[d]} \abs{\widehat{\mu^{\ba}}(\xi)}^{2} \, d\xi\, d\ba\\
		 &\leq \widetilde{C}\int_{\Sigma} \int_{\Sigma}  \int_{\euclid[d]} (1+\|{T_{x\wedge y}^{\ast} \xi}\|)^{-N}  \, d\xi  \,d\mu(x)d\mu(y) \\
		& \leq C'  \int_{\Sigma} \int_{\Sigma}  g_0(x\wedge y)^{-1}  \,d\mu(x)d\mu(y)  & (\text{by Proposition \ref{lem:t}}) \\
 		& \leq C' \sum_{n=0}^{\infty} \sum_{I\in\Sigma_{n}} g_0(I)^{-1} \mu([I])^2\\
 		& \leq C''  \sum_{n=0}^{\infty} \sum_{I\in\Sigma_{n}} r^{n} \mu([I])= \frac{C''}{1-r} < \infty. & (\text{by } \eqref{e-ttt*}) \\
 			\end{aligned}
	\end{equation*}
It implies that for $ \mathcal L^{md} $-a.e.\ $ \ba \in \euclid[md] $, $\widehat{\mu^\ba}\in L^2(\R^d)$ and thus  $\mu^\ba$ is absolutely continuous with an $L^2$-density.
\end{proof}

\bigskip

\noindent {\bf Acknowledgements}. This work was partially supported by the General Research Funds (CUHK14304119, CUHK14303021)  from the Hong Kong Research Grant Council, and by a direct grant for research from the Chinese University
of Hong Kong. The authors are grateful to Changhao Chen for helpful comments.

\end{document}